\documentclass[12pt,reqno]{article}
mm \textheight=8.9in

\usepackage{amssymb}
\usepackage{amsmath}
\usepackage{amsthm}
\usepackage{color}
\usepackage{graphics,shortvrb}
\usepackage{graphicx}

\newcommand{\br}[3]{{$#1$}$\lower4pt\hbox{$\tp\atop\raise4pt \hbox{$\scriptscriptstyle{#2}$}$} ${$#3$}}
\newcommand{\tw}[3]{{$#1$}${\,\scriptscriptstyle {#2}}\atop\raise9pt\hbox{$\scriptstyle\tp$} ${$#3$}}
\newcommand{\ttps}[2]{{#1}\raise5pt\hbox{$\lower12pt\hbox{$\scriptstyle\tp$}\atop \lower0pt\hbox{$\tilde\;$}$}\raise4.5pt\hbox{${\scriptstyle{#2}}$}}
\newcommand{\st}[1]{\mbox{${\,\scriptscriptstyle {#1}}\atop\raise5.5pt\hbox{$*$}$}}

\newcommand{\rd}[1]{\mbox{${\,\scriptscriptstyle {#1}}\atop\raise5.5pt\hbox{$\bullet$}$}}
\newcommand{\rt}[1]{\otimes_\chi}
\newcommand{\lt}[1]{\mbox{${\,\scriptscriptstyle {#1}}\atop\raise5.5pt\hbox{$\ltimes$}$}}
\newcommand{\btr}{\raise1.2pt\hbox{$\scriptstyle\blacktriangleright$}\hspace{2pt}}
\newcommand{\btl}{\raise1.2pt\hbox{$\scriptstyle\blacktriangleleft$}\hspace{2pt}}

\newcommand{\lcr}{\raise1.0pt \hbox{${\scriptstyle\rightharpoonup}$}}
\newcommand{\rcr}{\raise1.0pt \hbox{${\scriptstyle\leftharpoonup}$}}

\newcommand{\ttp}{{\lower12pt\hbox{$\tp$}\atop \hbox{$\tilde\;$}}}

\newcommand{\gr}{\mathrm{gr}\>}
\newcommand{\id}{\mathrm{id}}

\newcommand{\Tc}{\mathcal{T}}

\renewcommand{\S}{\mathcal{S}}
\newcommand{\Ru}{\mathcal{R}}

\newcommand{\Vc}{\mathcal{V}}
\newcommand{\Wc}{\mathcal{W}}

\newcommand{\Q}{\mathcal{Q}}

\newcommand{\Nc}{\mathcal{N}}

\newcommand{\C}{\mathbb{C}}
\newcommand{\Z}{\mathbb{Z}}

\newcommand{\tp}{\otimes}

\newcommand{\V}{V}
\newcommand{\U}{U}

\newcommand{\F}{\mathcal{F}}
\newcommand{\ve}{\varepsilon}

\newcommand{\dt}{\delta}

\newcommand{\op}{\oplus}
\newcommand{\la}{\lambda}

\newcommand{\Char}{\mathrm{ch }}

\newcommand{\End}{\mathrm{End}}

\newcommand{\Span}{\mathrm{Span}}

\newcommand{\mlt}{\mathrm{mlt}}
\newcommand{\Tr}{\mathrm{Tr}}

\newcommand{\Rm}{\mathrm{R}}

\newcommand{\diag}{\mathrm{diag}}

\newcommand{\g}{\mathfrak{g}}
\renewcommand{\b}{\mathfrak{b}}
\renewcommand{\k}{\mathfrak{k}}
\newcommand{\h}{\mathfrak{h}}

\newcommand{\s}{\mathfrak{s}}

\renewcommand{\o}{\mathfrak{o}}

\newcommand{\m}{\mathfrak{m}}

\newcommand{\nn}{\nonumber}

\renewcommand{\l}{\mathfrak{l}}

\newcommand{\si}{\sigma}
\newcommand{\al}{\alpha}

\newcommand{\bt}{\beta}

\newcommand{\be}{\begin{eqnarray}}
\newcommand{\ee}{\end{eqnarray}}

\newtheorem{thm}{Theorem}[section]
\newtheorem{propn}[thm]{Proposition}
\newtheorem{lemma}[thm]{Lemma}
\newtheorem{corollary}[thm]{Corollary}

\newtheorem{remark}[thm]{Remark}

\newcount\prg

\newcommand{\parag}{\advance\prg by1 {\noindent\bf\thesection.\the\prg\hspace{6pt}}}

\begin{document}
\title{Quantum exceptional group  $G_2$ \\and its conjugacy classes}
\author{
Alexander Baranov$^\dag$, Andrey Mudrov$^{\dag,}$\footnote{This study is supported in part by the RFBR grant 15-01-03148.} \hspace{1pt}, and Vadim Ostapenko$^\ddag$\vspace{20pt}\\
{\em \small  Dedicated to the memory of Petr Kulish}
\vspace{10pt}\\
$\dag$ \small Department of Mathematics,\\ \small University of Leicester, \\
\small University Road,
LE1 7RH Leicester, UK\\[10pt]
$\ddag$ \small Department of Mathematics,\\ \small Bar Ilan University, \\
\small Ramat Gan
 5290002, Israel.
 \\
%\small e-mail: am405@le.ac.uk\\
%[0.1in]\select{DRAFT}
}
\date{}
\maketitle
\begin{abstract}
  We construct quantization of semisimple conjugacy classes of the exceptional group $G=G_2$ along with and by means of their
exact representations in highest weight modules of the quantum group $U_q(\g)$. With every point $t$ 
of a fixed maximal torus we associate a highest weight module $M_t$ over $U_q(\g)$ and realize the quantized polynomial algebra
of the class of $t$ by linear operators on $M_t$. Quantizations corresponding to points of the same orbit of the Weyl group
are isomorphic.
\end{abstract}

{\small \underline{Mathematics Subject Classifications}: 81R50, 81R60, 17B37.
}

{\small \underline{Key words}: Quantum groups, conjugacy classes, quantization
}
\newpage
\tableofcontents
\section{Introduction}
Exceptional Lie groups occupy a special position in mathematics among simple groups and find important applications in theoretical physics
in connection with string theories, supergravity, and grand unification, \cite{Ra,HS}. At the same time their quantum analogs are the least studied
compared to other quantum groups.
In the this paper we focus on the simplest exceptional group $G=G_2$ and 
construct quantization of its semisimple conjugacy classes via operator algebra realization.

To this end, we undertake a thorough analysis of certain modules over the quantum group $U_q(\g)$, where $\g$ is the Lie algebra of $G$.
In particular, we associate such module, $M_\la$, with every point $t$ of a fixed maximal torus of $G$ and prove that the polynomial
algebra of the class $O_t\ni t$ can be quantized as a subalgebra in $\End(M_\la)$. Here $\la$ is the highest weight of $M_\la$, that
depends on $t$. The quantum polynomial algebra
is simultaneously presented as a quotient of the locally finite part $\C_q[G]$ of the adjoint module $U_q(\g)$ by the explicitly given ideal annihilating $M_\la$.

The quantization is facilitated by the properties of the matrix $\Q\in U_q(\g)\tp U_q(\g)$ expressed through the 
universal R-matrix $\Ru$ as $\Q=\Ru_{21}\Ru$. It is known to commute with the coproduct $\Delta(x)$ 
of all elements $x\in U_q(\g)$ and satisfy the reflection equation \cite{DKM}. On specialization of the  left tensor factor 
to the minimal representations of $U_q(\g)$ on $V=\C^7$, the entries of matrix $\Q$ generate $\C_q[G]\subset U_q(\g)$. Quantization is possible due to semisimplicity of the operator $\Q$ on $V\tp M_\la$, which 
 is a consequence of direct sum decomposition of  $V\tp M_\la$ into a sum of submodules of highest weight.
 It may be  interpreted as quantization of the basic quantum homogeneous vector bundles over $O_t$.
  We
give exact criteria for such decomposition for each module $M_\la$, which  required a detailed study of singular vectors in
$V\tp M_\la$.

The setup of the paper is a follows.
Section \ref{secCCC} presents a classification of semisimple conjugacy classes of $G_2$ adopted to our purposes.
It is followed by the basic information about the quantum group $U_q(\g)$
and its minimal representation on $\C^7$ in Section \ref{secQUEA}.
The quantization theorem based on decomposition of $V\tp M_\la$ is stated in Section
\ref{secQUEA}. The subsequent sections prove this decomposition for each type of $M_\la$.
In Section \ref{secGVM} we define generalized parabolic Verma modules and establish some  their  properties.
We do regularization of singular vectors in $V\tp M_\la$ for a general Verma module $M_\la$,
in Section \ref{secModStrH}.
The last three sections are devoted to regularization  of singular vectors
in tensor products and direct sum decomposition of $V\tp M_\la$ for all other types of $M_\la$.
Some useful formulas including the entries of matrix participating in the Shapovalov inverse form 
can be found in Appendix.

Throughout the paper we adopt the following general convention:
\begin{itemize}
  \item For better readability of formulas, we denote a scalar inverse by the bar, e.g. $\bar q=q^{-1}$.
  \item The notation $a\simeq b$ means  that $a$ is proportional to $b$ with a non-zero scalar factor. If the coefficient is
a scalar function, we thereby assume that it never turns zero. The symbol $\simeq$ also stands for isomorphism, which
 is always clear from the context and causes no confusion.
  \item Divisibility by a regular scalar function $\phi$ is denoted by $\phi \sqsubset$.
\end{itemize}
\section{Semisimple conjugacy classes of $G_2$}
\label{secCCC}
In this section we describe semisimple conjugacy classes of the complex algebraic group $G=G_2$.
Let  $\g$ denote the Lie algebra of $G$ and with a fixed Cartan subalgebra $\h$.
Its root system $\Rm$ is displayed on the figure below.
\vspace{10pt}

\hspace{150pt}
\includegraphics[width=140pt,height=142pt]{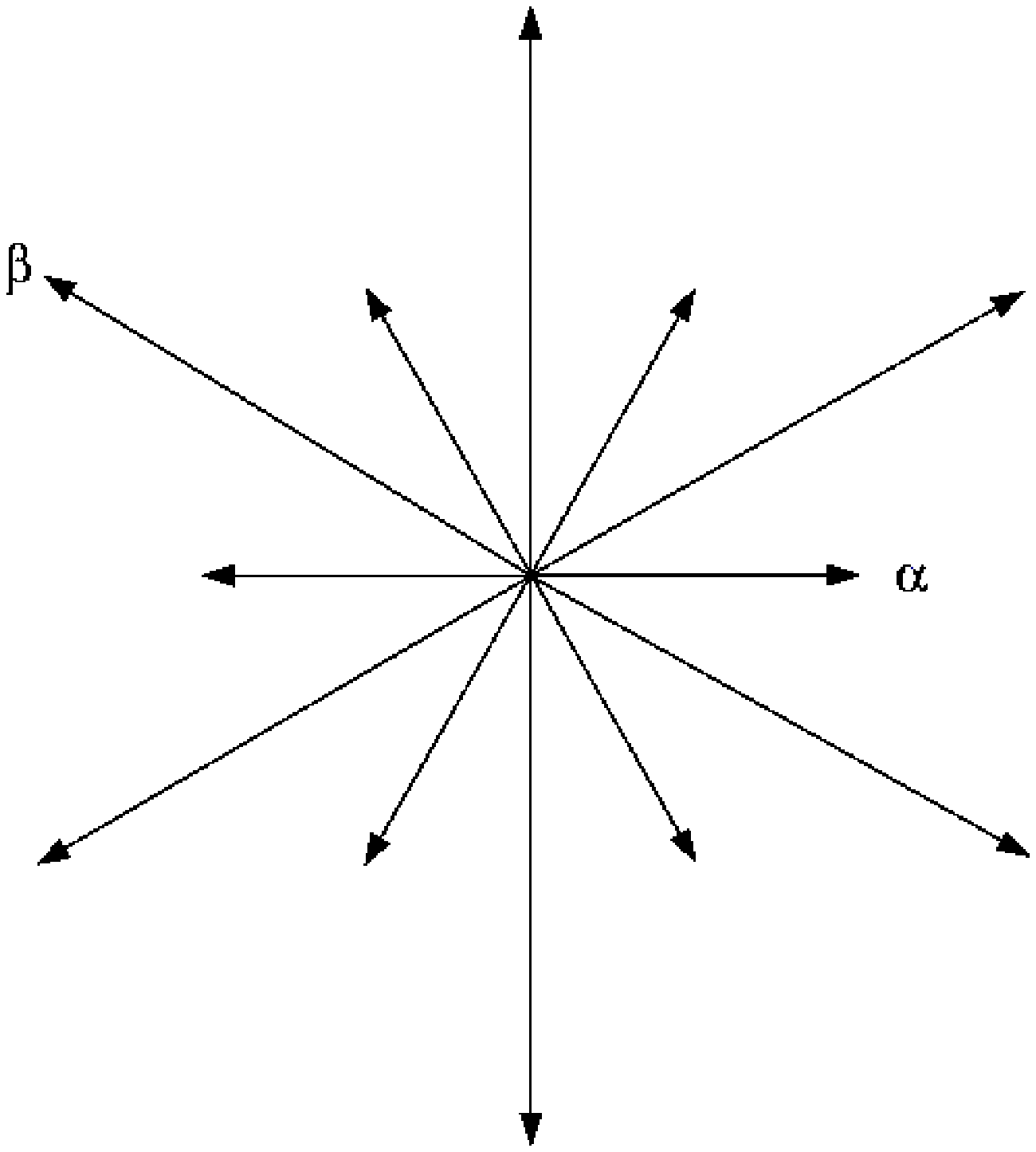}
\vspace{5pt}

\noindent
The subset of positive roots $\Rm^+\subset \Rm$ contains the basis of simple roots $\Pi=\{\al,\bt\}$.
Fix an inner product on $\h^*$ as
$$
(\al,\al)=2, \quad (\al,\bt)=-3, \quad  (\bt,\bt)=6.
$$
The half-sum of positive roots $5\al+3\bt$  is denoted by $\rho$.

To all $\la\in \h^*$ we assign its image $h_\la$  under the isomorphism $\h^*\simeq \h$
via the canonical form: $\mu(h_\la)=(\la,\mu)$ for all $\mu\in \h^*$.

The group $G$ has an exact representation in $\C^7$.
The corresponding representation of the quantum group is given in Section \ref{Rep}.

The affine Dynkin diagram
$
\>
\begin{picture}(45,10)
  \put(0,3){\circle*{4}} \put(1,3){\line(1,0){18}}
  \put(20,3){\circle{2}}  \put(21,3){\line(1,0){18}}
  \put(40,3){\circle{2}}
  \put(32,0){$>$}\put(21,4){\line(1,0){16}}\put(21,2){\line(1,0){16}}
  \end{picture}
$
of  $\g$
suggests the following stabilizers $\k\subset \g$ of semisimple conjugacy classes labelled by the their root bases $\Pi_\k\subset \Rm^+_\g$:

\begin{center}
\begin{tabular}{|c|ccc|}
  \hline
  $\k$ & &$\Pi_\k$ &   \\
  \hline
  $\h$ & &$\varnothing$ &   \\
  \hline
    $\k_s$ & $\{\al\}$,& $\{\al+\bt\}$,& $\{2\al+\bt\}$   \\
  \hline
  $\k_l$ &  $\{\bt\}$, & $\{3\al+\bt\}$,& $\{3\al+2\bt\}$   \\
  \hline
  $\k_{s,l}$ &  $\{\al,3\al+2\bt\}$, &$\{\al+\bt,3\al+\bt\}$,& $\{2\al+\bt,\bt\}$  \\
  \hline
  $\k_{l,l}$ &  &$\{\bt,3\al+2\bt\}$&  \\
  \hline
\end{tabular}
\end{center}
The subscripts indicate the lengths of roots. There are three Levi  types with $\#\Pi_\k\leqslant 1$ and two pseudo-Levi types with $\#\Pi_\k=2$.

Different although isomorphic $\k$ give rise to the same
conjugacy class $G/K$, where $K$ is the subgroup with the Lie algebra $\k$.
Still we make this distinction because we associate with them different representations
of quantized $G/K$.

Let $T$ denote the maximal torus of $G$ corresponding to $\h$ and fix $t\in T$
such that $\k$ is the centralizer of $t$. We parameterize $T$ with a pair of complex coordinates $x,y\in \C^\star$.
In the matrix realization that gives
$$
t=\diag(xy,x,y,1,\bar x,\bar y,\bar x\bar y)\in \End(\C^7).
$$
Regarding the roots as  characters on $T$, we have $\al(t)=y$, $\bt(t)=xy^{-1}$.

Define $T^{\k}\subset  T$ as the subset of points whose centralizer Lie algebra is $\k$.
We will also use the notation $T^{\Pi_\k}=T^{\k}$. We select the subset $T^\k_{reg} \in T^\k$  of regular points,
whose minimal polynomial in the basic representation has maximal degree. The complementary subset
in $T^\k$ is denoted by $T^\k_{bord}$ and called borderline, following  \cite{AM}. Such points
are present only for $\k=\h$ and $\k=\k_l$. They are a sort of "transitional" from Levi to pseudo-Levi type, hence the name.

The set $T^\varnothing=T^\h$ is determined by the conditions
$x\not =1$, $y\not =1$, $x\not =y$, $xy^2\not =1$, $x^2 y\not =1$, $x y\not =1$.
It is convenient to use the first three diagonal matrix entries for description of  $t$:
\be
&&\begin{array}{rcccccccc}
T^\varnothing_{reg}: (xy,x,y), &x\not =y, &x^iy^j\not =1, &i,j=0,1,2;
\end{array}
\label{t-h-reg}\\
&&\begin{array}{ccccccc}
T^\varnothing_{bord}:(-x,-1,x), \quad (-x,x,-1),\quad (-1,x,-\bar x),\quad x^4 \not =1;
\end{array}
\label{t-h-bord}\\
&&\begin{array}{ccccccc}
T^{\al}: (x,x,1), &T^{\al+\bt}:(x,1,x),&T^{2\al+\bt}:(1,x,x^{-1}),&x^2\not =1;\\
\end{array}
\label{t-s}\\
&&\begin{array}{ccccccc}
T^{\bt}_{reg}:(x^2,x,x), &T^{3\al+\bt}_{reg}:(\bar x, \bar x^2,x), &
T^{3\al+2\bt}_{reg}:(\bar x,x,\bar x^2), &x^3,x^4\not =1; \\
\end{array}
\label{t-l-reg}\\
&&\begin{array}{ccccccc}
T^{\bt}_{bord}:(-1,x,x), &T^{3\al+\bt}_{bord}:(\bar x, -1,x), &
T^{3\al+2\bt}_{bord}:(\bar x,x,-1), &x=e^{\pm \frac{\pi i}{2}}; \\
\end{array}
\label{t-l-bord}\\
&&\begin{array}{ccccccc}
T^{\al,3\al+2\bt}: (-1,-1,1), &
T^{\al+\bt,3\al+\bt}:(-1, 1,-1),&
T^{2\al+\bt,\bt}:(1,-1,-1);&\\
\end{array}
\label{t-sl}\\
&&\begin{array}{ccccccc}
&T^{\bt,3\al+\bt }:(\bar x,x,x), & x =e^{\pm \frac{2\pi i}{3}}. \\
\end{array}
\label{t-ll}
\ee
We have listed all possible  $\k\subset \g$, so that the sets
$$
T^{s}=T^{\al}\cup T^{\al+\bt} \cup T^{2\al+\bt}, \quad T^{l}=T^{\bt}\cup T^{3\al+\bt} \cup T^{3\al+2\bt},
$$
$$
T^{{s,l}}=T^{\al,3\al+2\bt}\cup T^{\al+\bt,3\al+\bt} \cup T^{2\al+\bt,\bt}, \quad T^{{l,l}}=T^{\bt,3\al+\bt },
$$
along with $T^\varnothing$ exhaust all of $T$. They consist of points whose
conjugacy classes are isomorphic as homogeneous spaces.

Denote by $\tilde G$ the group $SL(7)$ and by $\tilde W$ its Weyl group. The Weyl group of $G$ is denoted by $W$.
It is elementary fact  that the intersection of $\tilde G$-conjugacy class with $G$ consists of a finite number of $G$-classes.
\begin{propn}
The conjugacy class of each semisimple point $t\in G$ is the  intersection of its  $\tilde G$-conjugacy class with $G$.
\end{propn}
\begin{proof}
A semisimple conjugacy class of $\tilde G$ is determined by the set of eigenvalues and their multiplicities.
It is sufficient to check that $\tilde W t \cap T= Wt$ for all $t\in T$.

Fix  $t=(xy,x,y)\in T^\varnothing$ so that $y\not =-1$. One can assume that the multiplicity of $y$ is $1$, since otherwise
 $x\not =-1$ and one can flip $x\to y$ by $\si_\bt$.
Present $\tilde W t\cap T$  as a union $Y\cup \bar Y$ of sets whose elements have either $y$ or $\bar y$
among their first three coordinates. It is  disjoint for $t\in T^\varnothing$.

One can check that $\#|Y|=\#|\bar Y|=6$. Then $Y\subset Wt$ since
$$(xy,x,y) \stackrel{\si_{\bt}}{\sim} (xy,y,x) \stackrel{\si_{\al}}{\sim} (y,xy,\bar x)\stackrel{\si_{\bt}}{\sim}  (y,\bar x, xy) \stackrel{\si_{\al}}{\sim} (\bar x,y,\bar x\bar y) \stackrel{\si_{\bt}}{\sim}  (\bar x,\bar x\bar y,y).$$
The set $\bar Y$ is obtained from $Y$ by inverting the coordinates. One has  $\bar Y\subset Wt$ as
 $Y\ni (xy,x,y) \stackrel{\si_{\al}}{\sim} (x,xy,\bar y)\in \bar Y$.

Further,
$$
 T^{s}\ni(x,1,x) \stackrel{\si_{2\al+\bt}}{\sim}  (\bar x,\bar x,1)\stackrel{\si_{\bt}}{\sim} (\bar x,1,\bar x)\stackrel{\si_{2\al+\bt}}{\sim} (x,x,1)\stackrel{\si_{\al+\bt}}{\sim} (1,\bar x, x)\stackrel{\si_{\bt}}{\sim}  (1, x,\bar x),
 \quad x^2 \not =1,
$$
$$
 T^{l}\ni (\bar x, x,\bar x ^2) \stackrel{\si_{\bt}}{\sim} (\bar x,\bar x^2, x) \stackrel{\si_{\al}}{\sim} (\bar x^2,\bar x,\bar x) \stackrel{\si_{\bt+2 \al}}{\sim} (x^2,x,x)\stackrel{\si_{\al}}{\sim} (x,x^2,\bar x)\stackrel{\si_{\bt}}{\sim} (x,\bar x,x^2),\quad
 x^2,x^3 \not =1,
$$
$$
T^{l,s}\ni (1,-1,-1)\stackrel{\si_{\al}}{\sim} (-1,1,-1) \stackrel{\si_{\bt}}{\sim} (-1,-1,1),
$$
$$
T^{l,l}\ni (e^{\frac{2\pi i}{3}},e^{-\frac{2\pi i}{3}},e^{-\frac{2\pi i}{3}})\stackrel{\si_{\al}}{\sim}(e^{-\frac{2\pi i}{3}},e^{\frac{2\pi i}{3}},e^{\frac{2\pi i}{3}}).
$$
This proves $\tilde W t= Wt$ for all $t\in T$.
\end{proof}
\begin{corollary}
  The ideal of a semisimple conjugacy class in $\C[G]$ is generated by the entries of the minimal polynomial over the maximal ideal of the
  subalgebra of invariants.
\end{corollary}
\begin{proof}
  Fix a semisimple point $t\in G$ and consider its conjugacy classes $\tilde O_t\subset  \tilde G$ and $O_t\subset G$.
  Let $F_1$ and $F_2$ be $G$-submodules in $\End(\C^7)$ generating the ideals $N(\tilde O_t)$ and $N(\tilde G)$, respectively.
  Put $f_i\colon \End(\C^7)\to F_i$ be the corresponding maps and set $f=f_1\oplus f_2$.
  By construction, $f$ is $G$-equivariant.
  It is sufficient to prove that $\ker df(t)=\ker df_1(t)\cap \ker df_2(t)$ has the same dimension as
  $O_t$, cf. \cite{M3}, Prop. 2.1.
  For the general linear group, $\tilde \m_t=\tilde \g\ominus\tilde \k=\ker df_1(t)$, and $\tilde \m_t=\mathrm{Ad}_t(\tilde \m_t)$.
  On the other hand, $\ker df_2(t)= \mathrm{Ad}_t(\g)$, and $\ker df=\ker df_1\cap \ker df_2=\tilde \m_t\cap \g=\g\ominus\k$
  since $t\in G$. This completes the proof.
\end{proof}
Recall from linear algebra that a semisimple $\tilde G$-class is determined by the set of eigenvalues and their  multiplicities.
Eigenvalues are fixed by the minimal polynomial while the multiplicities by the character of the subalgebra
of invariants under conjugation. The subalgebra of invariants when restricted to
maximal torus  $\tilde T\subset \tilde G$ is
generated by the functions $t\mapsto \Tr(t^m)$, $m\in 1,\ldots,7$, and the character is  evaluation at $t$.
The possible minimal polynomials of $t\in T $ are listed here:
$$
(t-xy)(t-x)(t-y)(t-1)(t-\bar x)(t-\bar y)(t-\bar x \bar y), \quad t\in T^\varnothing_{reg}, \quad (x,y)\in \C^2_{reg},
$$
$$
(t^2-x^2)(t^2-1)(t^2-\bar x^2),\quad t\in T^\varnothing_{bord}, \quad  x^4 \not =1, \quad \mlt(-1)=2, \quad
$$
$$
(t-x)(t-1)(t-\bar x),\quad t\in T^{s}, \quad   x^2 \not =1,\quad  \quad \mlt(x^{\pm 1})=2,
$$
$$
(t-x^2)(t-x)(t-1)(t-\bar x)(t-\bar x^2),\quad t\in T^{l}_{reg}, \quad  x^3 \not =1\not =   x^4,\quad \mlt(x^{\pm 1})=2,
$$
$$
(t^4-1),\quad t\in T^{s}_{bord}, \quad \mlt(-1)=\mlt(\pm i)=2,
$$
$$
(t^2-1),\quad t\in T^{s,l},\quad \mlt(-1)=4,
$$
$$
(t^3-1),\quad t\in T^{l,l},\quad \mlt(e^{\pm\frac{2\pi i}{3}})=3.
$$
Remark that  regular points, contrary to borderline, separate irreducible $\k$-submodules in $\C^7$.
The two bottom lines correspond to the two pseudo-Levi classes.
%\end{document}

\section{Quantized universal enveloping algebra}
\label{secQUEA}
%{Its dual $h_\rho\in \h$ is the unique solution of the  system of linear equations
%$\al_i(h_\rho)=\frac{1}{2}(\al_i,\al_i)$, $i=1,\ldots ,\rk\: \g$}.

Throughout the paper we assume  that $q\in \C$ is not a root of unity. Denote by $U_q(\g_\pm)$ the $\C$-algebra generated by  $e_{\pm\al}, e_{\pm\bt}$
 subject to the q-Serre relations
$$
e_{\pm \al}^{4}
e_{\pm \bt}
-
[4]_q
e_{\pm \al}^{3}
e_{\pm \bt}
e_{\pm \al}
+
[3]_q(q^2+q^{-2})
e_{\pm \al}^{2}
e_{\pm \bt}
e_{\pm \al}^2
-
[4]_q
e_{\pm \al}
e_{\pm \bt}
e_{\pm \al}^3
+
e_{\pm \bt}
e_{\pm \al}^4
=0
,
$$
$$
e_{\pm \bt}^{2}
e_{\pm \al}
-
(q^3+q^{-3})
e_{\pm \bt}
e_{\pm \al}
e_{\pm \bt}
+
e_{\pm \al}
e_{\pm \bt}^2
=0.
$$
Here and further on, $[z]_q=\frac{q^z-q^{-z}}{q-q^{-1}}$ whenever $q^{\pm z}$ make sense.

Denote by $U_q(\h)$ the commutative $\C$-algebra generated by $\{q^{\pm h_\al}\}_{\al\in \Pi}$, with $q^{h_\al}q^{- h_\al}=1$. The quantum group $U_q(\g)$ is a $\C$-algebra generated by  $U_q(\g_\pm)$ and $U_q(\h)$ subject
to the relations \cite{D}:
$$
q^{h_\al}e_{\pm\al}q^{h_\al}=q^{\pm 2} e_{\pm\al}, \quad q^{h_\al}e_{\pm\bt}q^{-h_\al}=q^{\mp 3} e_{\pm\bt},
$$
$$
q^{h_\bt}e_{\pm\al}q^{-h_\bt}=q^{\mp 3} e_{\pm\al}, \quad q^{h_\bt}e_{\pm\bt}q^{-h_\bt}=q^{\pm 6} e_{\pm\al},
$$
$$
[e_{\al},e_{-\al}]=\frac{q^{h_\al}-q^{-h_\al}}{q-q^{-1}}, \quad [e_{\bt},e_{-\bt}]=\frac{q^{h_\bt}-q^{-h_\bt}}{q^3-q^{-3}}.
$$
Remark that the vector space $\h$ is not contained in $U_q(\g)$, still it is convenient to keep reference to $\h$
for additive parametrization of monomials in $U_q(\h)$.

Set up the comultiplication on the generators  as 
\be
&\Delta(e_{\al})=e_{\al}\tp 1+ q^{h_{\al}}\tp e_{\al},
\quad
\Delta(e_{-\al})=1\tp e_{-\al} + e_{-\al} \tp q^{-h_{\al}},
\nn\\&
\Delta(q^{\pm h_{\al}})=q^{\pm h_{\al}}\tp q^{\pm h_{\al}},
\nn
\ee
for all $\al \in \Pi$. It is opposite to that in \cite{CP}.

We will use the notation $f_\al=e_{-\al}$, $f_\bt=[3]_qe_{-\bt}$, so the only  relation that is
inhomogeneous in $e_{-\al}$  translates to
$[e_{\bt},f_{\bt}]=\frac{q^{h_\bt}-q^{-h_\bt}}{q-q^{-1}}$.

The subalgebras in $U_q(\g)$ generated by $U_q(\g_\pm)$ over $U_q(\h)$ are quantized universal enveloping algebras of the
Borel subalgebras $\b_\pm=\h+\g_\pm\subset \g$ denoted further by $U_q(\b_\pm)$.

The Chevalley  generators $e_{\pm\al}$ can be supplemented with higher root vectors $e_{\pm\bt}$ for
all $\bt\in {\rm R}^+$. They participate in construction of a Poincar\'{e}-Birkhoff-Witt (PBW) basis in  $U_q(\g)$
and universal $\Ru$-matrix,  \cite{CP}.

The universal R-matrix is an element of a certain extension of $U_q(\g)\tp U_q(\g)$.
Let $\{\ve_i\}_{i=1}^2$ be an orthogonal basis in $\h^*$.
The exact expression for $\Ru$  up to the flip is  extracted from \cite{CP}, Theorem 8.3.9:
\be
\Ru= q^{\sum_{i=1}^2 h_{\ve_i}\tp  h_{\ve_i}}\prod_{\mu\in \Rm^+} \exp_{q_\mu}(1-q_\mu^{-2})(e_{-\mu}\tp e_{\mu} ) \in U_q(\b_-)\hat \tp U_q(\b_+),
\label{Rmat}
\ee
where  $\exp_{q}(x )=\sum_{k=0}^\infty q^{\frac{1}{2}k(k+1)}\frac{x^k}{[k]_q}$, $q_\mu=q^{\frac{(\mu,\mu)}{2}}$, and the product is
ordered in a certain way. Its reduction to the minimal representation can be found in \cite{Re,Se}.

\subsection{Minimal representation of $U_q(\g)$}
\label{Rep}
In this section we describe a representation of $U_q(\g)$ in the vector space $\C^7$.
It is a deformation of the classical representation of $\g$ restricted from $\s\o(7)$.
Our realization is close to \cite{Ku}.

Let $I$ denote the set of integers from $1,\ldots, 7$. Fix a basis $\{w_i\}_{i\in I} \in  \C^7=V$ and let $e_{ij}\in \End(V)$ denote the standard matrix units, $e_{ij}w_k=\dt_{jk}w_i$, $i,j,k\in I$.
One can check that the assignment
$$
\hspace{-38pt}\begin{array}{cclccl}
q^{h_\al}&\mapsto& qe_{11}+q^{-1}e_{22}+q^{2}e_{33}+e_{44}+q^{-2}e_{55}+q^{}e_{66}+q^{-1}e_{77},
\nn\\
q^{h_\bt}&\mapsto &e_{11}+q^{3}e_{22}+q^{-3}e_{33}+e_{44}+q^{3}e_{55}+q^{-3}e_{66}+e_{77},
\end{array}
$$
$$
\begin{array}{cclccl}
e_{\al}&\mapsto &e_{12}+e_{34}+e_{45}+e_{67},
&
f_{\al}&\mapsto &e_{21}+[2]_qe_{43}+[2]_qe_{54}+e_{76},
\nn\\
e_{\bt}&\mapsto &e_{23}+e_{56},
&
f_{\bt}&\mapsto &e_{32}+e_{65}.
\end{array}
$$
is compatible with the defining relations  and extends to a homomorphism $U_q(\g)\to \End(V)$.
%compatible with the standard triangular decomposition of $\End(V)$.

Up to scalar multiplies, the action of $U_q(\g_-)$ can be depicted by the graph by
\begin{center}
\begin{picture}(310,50)
\put(155,40){$\g_-$}

\put(35,20){$f_{\al}$}
\put(85,20){$f_{\bt}$}
\put(135,20){$f_{\al}$}
\put(185,20){$f_{\al}$}
\put(235,20){$f_{\bt}$}
\put(285,20){$f_{\al}$}

\put(5,0){$w_{7}$}
\put(55,0){$w_{6}$}
\put(105,0){$w_{5}$}
\put(155,0){$w_{4}$}
\put(205,0){$w_{3}$}
\put(255,0){$w_{2}$}
\put(305,0){$w_{1}$}

\put(10,15){\circle{3}}
\put(60,15){\circle{3}}
\put(110,15){\circle{3}}
\put(160,15){\circle{3}}
\put(210,15){\circle{3}}
\put(260,15){\circle{3}}
\put(310,15){\circle{3}}
\put(55,15){\vector(-1,0){40}}
\put(105,15){\vector(-1,0){40}}
\put(155,15){\vector(-1,0){40}}
\put(205,15){\vector(-1,0){40}}
\put(255,15){\vector(-1,0){40}}
\put(305,15){\vector(-1,0){40}}
\end{picture}
\end{center}
The representation of $U_q(\g_+)$  is obtained by reversing the arrows.
Let  $\nu_i \in \h^*$ denote the weight of $w_i$, then
$$
\nu_1=2\al+\bt, \quad \nu_2=\al+\bt, \quad \nu_3=\al, \quad \nu_4=0,
 \quad \nu_5=-\al,  \quad \nu_6=-\al-\bt, \quad \nu_7=-2\al-\bt.
$$
For all $\nu\in \Z \Pi$ we denote by $P(\nu)$ the set of pairs $(i,j)\in I\times I$ such that $\nu_i-\nu_j= \nu$.
For each pair with $i<j$ there is a unique element $\psi\in U_q(\g_-)$ such that $w_j=\psi w_i$. In fact, it is a monomial in
$f_\al,f_\bt$.
Let $\psi^{ij}\in U_q(\g_-)$ be the monomial obtained from it by reversing the order of factors.

\section{Quantum conjugacy classes}
\label{QCC}
In this section we describe quantum semisimple conjugacy classes along the lines of \cite{M1} and  \cite{AM1}.
The construction is based on certain facts from representation theory to be established in the subsequent sections.

We regard  elements of $\h^*$ as characters of $T$ and elements of $T$ as spectral points of $U_q(\h)$ via
the correspondence the correspondence $t\colon q^{h_\al}\mapsto \al(t)$, for all $t\in T$.
Fix $t\in T$ and its stabilizer subalgebra $\k$. Choose the weight $\la$ from the
condition $q^{2\la}=tq^{2\rho_\k-2\rho}=t_q$ regarded as an equality in $T$ upon the identification $\h\simeq \h^*$ via the
inner product.
Consider a 1-dimensional representation $\C_\la$ of $U_q(\b_+)$ extending $t_q\colon U_q(\h)\to \C$: it is implemented by the assignment $q^{\pm h_\al}\mapsto q^{\pm (\la,\al)}$, $e_\al\mapsto 0$.
Let $M_\la=U_q(\g)\tp_{U_q(\g_+)}\C_\la$ denote the Verma  $U_q(\g)$-module of highest weight $\la$ with the canonical generator  $v_\la\in M_\la$.

Due to the special choice of $\la$, it satisfies the conditions $q^{2(\la+\rho,\al)}=q^{(\al,\al)}$ for all $\al\in \Pi_\k$,
and we denote by $\Theta^\k$ the set of all such weights.
There exist singular vectors $v_{\la-\al}\in M_\la$ annihilated by  $e_\al$, $\forall\al\in \Pi$.
Up to a scalar multiplier, they can be written explicitly as $\hat f_{1i}v_\la$ with $(1,i)\in P(\al)$,
where $\hat f_{ij}$ are matrix elements of the (reduced) Shapovalov inverse. The vectors $v_{\la-\al}$ generate submodules $M_{\la-\al}\subset M_\la$. Let $M^\k_\la$ denote the quotient module $M_\la/\sum_{\al\in \Pi_\k}M_{\la-\al}$.
%We call it quantizing module for the conjugacy class of $t$.
On transition to the $\C[[\hbar]]$-extension via $q=e^\hbar$, $M_\la^\k$ is free over $\C[[\hbar]]$, by Lemma \ref{M is free}.

Denote by $\Q=(\pi\tp \id)(\Ru_{21}\Ru)\in U_q(\g)$ a matrix with entries in the quantum group. It commutes with the coproduct of all elements from
$U_q(\g)$ and plays a key role in the theory. Its entries generate a subalgebra $\C_q[G]\subset U_q(\g)$, which is
an equivariant quantization of $\C[G]$. So $\Q$ is the matrix of "coordinate" functions on $G$. It satisfies the so called "reflection equation" rather
then the RTT-relations of the Hopf dual to $U_q(\g)$.

The operator $\Q$ is scalar on every submodule of highest weight in $V\tp M_\la^\k$ as well as a quotient module.
Its eigenvalues can be described as follows. In the classical limit, the module $V$ is completely reducible over $\k$. Let $I^\k\subset I$ denote the subset of indices of $\k$-highest vectors $w_i$.
 Then the eigenvalues are $x_j=q^{2(\la+\rho,\nu_j)+(\nu_j,\nu_j)-(\nu_1,\nu_1)+2(\rho,\nu_1)}$.
It is proved in the subsequent sections that $\Q$ is semisimple on $V\tp M_\la^\k$,  so its minimal polynomial
is
\be
\prod'_{j\in I^\k}(\Q-x_j)=0,
\label{min_pol}
\ee
where the prime means that only distinct eigenvalues count (coincidences occur only for borderline $t$).
One can check that the set $\{x_j\}_{j\in I^\k_j}$ is invariant under the shifted action of Weyl group
$\si\colon \la\mapsto \si(\la+\rho)-\rho$.

Put $\Tr_q(X)=\Tr(\pi(q^{2h_\rho})X)$ for a matrix $X$ with arbitrary entries.
Then the elements $\tau^m=\Tr_q(\Q^m)$, $m\in I$,  generate the center of $\C_q[G]$. Note
that $\tau^m$ are not independent, as the rank of $G$ is two.

Let $\chi_\la$ denote the central character of $\C_q[G]$ returning $\tau^m v_\la=\chi_\la(\tau^m)v_\la$,
\be
\label{char_V1}
\chi_\la(\tau^m)&=&
\sum_{i=1}^7 x_{\nu_i}^m
\prod_{\al\in \Rm_+}\frac{q^{(\la+\rho+\nu_i,\al)}-q^{-(\la+\rho+\nu_i,\al)}}{q^{(\la+\rho,\al)}-q^{-(\la+\rho,\al)}}
.
\ee
It is invariant under the shifted action of the Weyl group $W$ on $\h^*$.

%Denote by $O_t\subset G$ the conjugacy class of $t$.
\begin{thm}
  The image of $\C_q[G]$ in $\End(M_\la^\k)$ is an equivariant quantization of the polynomial algebra $\C[O_t]$. It is the quotient of $\C_q[G]$ by the ideal generated by the entries of the minimal polynomial over the kernel of $\chi_\la$.
  The algebra $\C_q[O_t]$ depends only on the $W$-orbit of the point $t\in T$.
\end{thm}
\begin{proof}
  The proof is similar to \cite{M1}. Denote by $\S$ the quotient of $\C_q[G]$ by the ideal generated by $\ker \chi_\la$.
  Denote by $\Tc$ the extension of $\End(M_\la^\k)$ over the Laurent series. Let $U_\hbar(\g)$ be the $\C[[\hbar]]$-extension of
  $U_q(\g)$ completed in $\hbar$-adic topology. The map $\S\to \Tc$ is $U_\hbar(\g)$-equivariant. The algebra $\S$ is a direct sum
  of $U_\hbar(\g)$-modules that are $\C[[\hbar]]$-finite (as embedded in the locally finite part of $\End(M_\la))$. Since $\Tc$ has no $\hbar$-torsion, the image of  $\S$ in $\Tc$ is $\C[[\hbar]]$-free. The relations (\ref{min_pol}) and (\ref{char_V1}) go over to the relations of $N(O_t)$,  hence
  they generate the kernel of the map $\C_q[O_t]\to \Tc$, cf. \cite{M1} for details.
  This kernel is invariant under the shifted action of $W$ on the weight $\la$, therefore,   $\C_q[O_t]$ depends on the class of $t$.
\end{proof}
The key fact the proof based upon is semisimplicity of $\Q$. To this end, it is sufficient to show that highest weight submodules exhaust all
of $V\tp M^\k_\la$. We solve a stronger problem: we establish exact criteria when $V\tp M^\k_\la$ splits into a direct sum of highest weight modules.
The rest of the paper is devoted to this analysis.
\section{Generalized parabolic Verma modules}
\label{secGVM}
Fix a weight  $\la\in \h^*$ and consider the Verma module $M_\la$.
Let ${}^*\!M_\la$ denote the Verma module of lowest weight $-\la$. There is a unique, up to a scalar multiple, $U_q(\g)$-invariant form $M_\la\tp {}^*\!M_\la\to \C$
(equivalent to the contravariant Shapovalov on $M_\la$), which is non-degenerate if and only if  $M_\la$ is irreducible.
As that is the case for generic weight, \cite{DCK}, there is a unique lift $\hat \F\in U_q(\g_+)\tp U_q(\b_-)$ of the inverse form, where
the Borel subalgebra is extended over the ring of fractions of $U_q(\h)$.
The matrix  $\hat F= (\pi\tp \id)(\hat \F)\in \End(V)\tp U_q(\b_-)$ is  expressed through a matrix $F\in \End(V)\tp U_q(\g_-)$
 whose entries $f_{ij}$ are presented in Section \ref{ShapInv}. 
 
 Put $\rho_i=(\rho,\nu_i)$,  $\tilde \rho_i=\rho_i+\frac{1}{2}||\nu_i||^2$, $i\in I$, and
define
$$
\eta_{ij}=h_{\nu_i}-h_{\nu_j}+\rho_i-\rho_j-\frac{||\nu_i-\nu_j||^2}{2}, \quad
A^j_i=\frac{q-q^{-1}}{1-q^{2\eta_{ij}}},\quad i<j.
$$
For an ascending sequence of integers  $m_1,\ldots, m_k,j$, put $f_{m_1,\ldots,m_k}=f_{m_1,m_2} \ldots  f_{m_{k-1},m_k}$ and $A^j_{m_1,\ldots,m_k}=A^j_{m_1}\ldots A^j_{m_k}.$  Then
\be
\hat f_{ij}=\sum_{i<\vec m< j}f_{i,\vec m,j}A^j_{i,\vec m}q^{\eta_{ij}-\tilde \rho_i+\tilde \rho_j}.
\label{hat_f}
\ee
where the summation is taken over all sequences $\vec m =(m_1,\ldots,m_k)$ such that $i<m_1$, $m_k<j$, including $\vec m=\varnothing$.
Finally,  $\hat f_{ii}=1$ for all $i$ and $\hat f_{ij}=0$ for $i>j$.

We regard $U_q(\h)$ as the algebra of trigonometric polynomials  on $\h^*$.
The linear isomorphism $U_q(\b_-)\simeq  U_q(\g_-)\tp  U_q(\h)$ identifies elements of $U_q(\b_-)$ with
 functions $\h^*\to  U_q(\g_-)$.
% Similarly we interpret  vectors from $M_\la$
%due to the natural lift $M_\la\to U_q(\g_-)$, $xv_\la \mapsto x$.

%Suppose that $\al\in \Pi_\k$ and  $\la\in \Theta^\k$.
The singular vector $v_{\la-\al} \in M_\la$ can be constructed as follows.
It is known to be unique, up to a scalar factor (follows from the basics properties of reduction algebras).
Therefore it is proportional to  $\hat f_{ij}v_{\la}=\hat f_{ij}(\la)v_{\la}$ with $(i,j)\in P(\al)$, upon an appropriate regularization
of $\hat f_{ij}(\la)$ if needed.
\begin{lemma}
Suppose that $t\in T^\k$ and $\la=\frac{1}{2}\ln t_q \in \h^*$ is fixed as in
Section \ref{QCC}. For all $\al\in \Pi_\k$ choose a pair $(i,j)\in P(\al)$.
Then $\hat f_{ij}(\la)\in U_q(\g_-)$ is a deformation of a classical root vector, $f_\al\in \g_-$.
\end{lemma}
\begin{proof}
  The proof is based on the fact that in any presentation of $\al$ as a sum of positive roots the
  summands do not belong to $R^+_\k$, see e.g. \cite{AM1}, Lemma 2.2.
\end{proof}
\begin{corollary}
\label{M is free}
  The $\C[[\hbar]]$-extension of $M_\la^\k$ is $\C[[\hbar]]$-free.
\end{corollary}
\begin{proof}
  The proof is similar to \cite{M3}, Proposition 6.2.
\end{proof}
\subsection{Standard filtration of $V\tp M_\la$}
Define $ \Vc_j\subset V\tp M_\la$, $j\in I$, to be the submodule generated by $\{w_i\tp v_\la\}_{i=1}^j$. They form an ascending filtration $\Vc_\bullet$
of $V\tp M_\la$, which we call standard.
Its  graded module $\gr \Vc_\bullet$ is the direct sum $\oplus_{i=1}^j\gr \Vc_j$, where $\gr \Vc_j=\Vc_j/\Vc_{j-1}$ is isomorphic to $M_{\la+\nu_j}$ for
all $\la$ (the proof is similar to  \cite{BGG} for classical $U(\g)$). It is generated by the image $w^j_\la$ of $w_j\tp v_\la$ in $\gr \Vc_j$.
\begin{propn}
\label{principal_diagram1}
Suppose $(i,j)\in P(\bt)$, and $\psi$ is a Chevalley monomial of weight $-\bt$.
If $\psi \not \simeq \psi^{ij}$, then $w_i\tp \psi v_\la \in  \Vc_{j-1}$.
Otherwise,
\be
w_i\tp \psi^{ij}v_\la &\simeq &w_{j}\tp v_\la \mod  \Vc_{j-1}.
\ee
%\blacksquare
\end{propn}
\begin{proof}
  The proof is similar to \cite{AM1}, Proposition 3.5.
\end{proof}

Assuming $\la\in \Theta^\k$ denote by $\Vc_j^\k$ the image of $\Vc_j$ under the projection $V\tp  M_\la\to V\tp  M_\la^\k$.
Clearly the sequence $\Vc_\bullet^\k=(\Vc_j^\k)$ forms an ascending filtration of $V\tp  M_\la^\k$.
Denote by $\bar I^\k$ the complement of $I^\k$ in $I$. Then $j\in \bar I^\k$ if and only if there is $i<j$ such that $\nu_i-\nu_j\in \Pi_\k$.

\begin{propn}
   The graded module  $\gr \Vc_\bullet^\k$ of the filtration $\Vc_\bullet^\k$ is isomorphic to $\op_{j\in I^\k}\gr \Vc_j^\k$.
\end{propn}
\begin{proof}
 Fix $j$ and put and $\bt=\nu_1-\nu_j$.
 The module $\gr \Vc_j^\k$ is a quotient of $\Vc_j$ by the submodule $\Vc_{j-1}+(V\tp M)\cap \Vc_j$, where $M=\sum_{\al\in \Pi_\k}M_{\la-\al}$.
 By Proposition \ref{principal_diagram1}, its subspace of weight $\nu_j+\la$ is isomorphic to  the quotient
 of $w_1\tp M_\la[\la-\bt]$ by $w_1\tp \Nc_\bt v_\la+w_1\tp M[\la-\bt]$, where $\Nc_\bt\subset U_q(\g_-)$ is spanned by  Chevalley monomials of weight $-\bt$
  that are not proportional to $\psi^{1j}$.
% Since $w_i\tp \hat f_{ij} v_\la \simeq w_i\tp \psi^{ij}v_\la$ modulo $\Vc_{j-1}$ for all $i,j$ and generic $\la$, it is true for all $\la$ upon
%  renormalization of $\hat f_{ij}v_\la$ if needed.
 Then $M_\la[\la-\bt]=\Nc_\bt v_\la+M[\la-\bt]$ if and only if $\psi^{1j}v_\la\in M[\la-\bt]\mod \Nc_\bt v_\la$, which is equivalent to $j\in \bar I^\k$. Otherwise $\gr \Vc_j^\k[\la-\bt]$ is spanned by the image of $w_1\tp \psi^{1j}v_\la\simeq w_j\tp v_\la$, the generator of $\gr \Vc_j^\k$.
 This proves that $\Vc_j^\k=\Vc_{j-1}^\k$ for each $j\in \bar I^\k$ and $\Vc_j^\k/\Vc_{j-1}^\k\not =\{0\}$ for $j\in I^\k$.
 \end{proof}

Let $M_j\subset V\tp M_\la$ denote the submodule of highest weight $\la+\nu_j$  and let  $u_j$ be its highest weight generator.
Furthermore, consider $M_\la^\k$ for $\la\in \Theta^\k$ and let $\pi^\k_V$ denote the projection $V\tp M_\la\to V\tp M_\la^\k$.
Define $M_j^\k=\pi^\k_V(M_j)^\k$ and $\Wc^\k_j=\sum_{i=1}^7 M_j^\k$. The ascending sequence $\Wc_\bullet^\k=(\Wc_j^\k)$, $j=1,\ldots,7$, of submodules is also
invariant under the action of $\Q$, which
is semisimple on $W_7^\k$. Semi-simplicity of $\Q$ is important for our studies, so the question is when $\Wc_7^\k=V\tp M^\k_\la$ or, more specifically, $\Wc_7^\k=\op_{j\in I^\k} M^\k_j$. We answer this question by comparing $\Wc_\bullet^\k$ against  $\Vc_\bullet^\k$.
First of all, observe that  $\Wc^\k_j\subset \Vc^\k_j$, by Proposition \ref{principal_diagram1}.
\begin{propn}
\label{W=V}
Suppose that $\k$ is Levi and fix $j\in I^\k$.
Then the following statements are equivalent:
{\em i)}  $\Vc_j^\k= \Wc_j^\k$,  {\em ii)}  $\Vc_i^\k=\Wc_i^\k$ all $i\leqslant j$,
{\em iii)} projection $\wp_i^\k\colon V\tp M^\k_\la\to \gr \Vc_i^\k$ is an isomorphism on $M_i^\k$ for all $i\leqslant j$,
 {\em iv)} $\Wc_j^\k=\op_{i=1}^j M_i^\k$.
\end{propn}
\begin{proof}
  It can be proved that, for  Levi $\k$, both $M_j^\k$ and $\gr \Vc_j^\k$ are parabolically induced from
  the same $U_q(\k)$-module.
  Hence the map $M_j^\k\to \gr \Vc_j^\k$ is epimorphism and isomorphism simultaneously unless it is zero.
    
  The implication ii) $\Rightarrow$ i) is trivial.
With $\Wc^\k_1=\Vc^\k_1$, assume that ii) is violated and let $k> 1$ be the smallest such that $\Wc_k^\k\not = \Vc_k^\k$.
Comparison of weight subspaces gives $\dim \Vc_j^\k[\la+\ve_j]=\dim \Wc_j^\k[\la+\ve_j]+1$ for all $j\geqslant k$, so  i) $\Rightarrow$ ii).
Assuming  ii) we find that all maps $M_i^\k\to \gr \Vc_i^\k$ are surjective and therefore injective; hence  iv). 
Conversely, iv) implies that all maps $\Wc_i^\k\to \gr \Vc_i^\k$ are surjective. Since, $\Wc^\k_1=\Vc_1^\k$, 
 induction on $i$ then proves ii). Furthermore, iv)  implies that $M_i^\k\to \gr \Vc_i^\k$ are isomorphisms, and then $M_i^\k\cap \Wc_{i-1}^\k\subset M_i\cap \Vc_{i-1}^\k =\{0\}$, which proves  iii). Finally, induction on $i$ yields iii) $\Rightarrow$ ii).
\end{proof}
A direct sum decomposition $V\tp M_\la^\k=\op_jM_j^\k$ implies that the operator $\Q$ is semisimple on $V\tp M_\la^\k$.
More generally, $\Q$ is semisimple if $V\tp M_\la=W_7^\k$. That is the case if all maps $\wp_j^\k\colon M_j^\k\to \gr \Vc_j^\k$ are onto,
i.e. the generators of $M_j^\k$ are not killed by $\wp_j^\k$.

\section{Module structure of $V\tp M_\la$}
\label{secModStrH}
In this section, where work out exact criteria for decomposition
of $\C\tp M_\la$ into a direct sum of submodules of highest weight. To that end, we undertake a detailed study of singular vectors
$\hat u_j=\hat F(w_j\tp v_\la)\in \C\tp M_\la$ as rational functions $\h^*\to V\tp U(\g_-)$ upon the natural identification of $M_\la$ with $U_q(\g_-)$ as a vector space.
We end up with rescaled singular vectors $u_j$, $j\in I$,  that are regular on $\h^*$ and never turn zero.
\subsection{Singular vectors in $V\tp M_\la$}
The  vectors $\hat u_j=\hat F(w_j\tp v_\la)$, $j\in I$, are expanded as
\be
\hat u_j=\sum_{i=1}^jw_i\tp \hat f_{ij}v_\la\in V\tp M_\la.
\label{hat u}
\ee
They are singular for all $\la$ where defined and generate submodules $M_j\subset \C\tp M_\la$ of highest weight $\la+\nu_j$.
They have rational trigonometric dependence on  $\la$ and may have zeros and poles.
As they matter up to scalar factors,  it is convenient to pass from  $\hat u_j$ to
$\check u_j=\bar A_{1,\ldots,j-1}^j(\la)\hat u_j $,
which are regular in $\la$.
Then
\be
\label{check_u}
\check u_j=\sum_{i=1}^{j}w_i\tp \check u_{ij},\quad \mbox{where}\quad \check u_{ij}=\bar A_{1,\ldots,i-1}^j(\la)\check f_{ij}v_\la, \quad \check f_{ij}=\hat f_{ij}\bar A_{i,\ldots,j-1}^j\in U_q(\b_-).
\ee
They  generate submodules $M_j\subset V\tp M_\la$ if do not turn zero, otherwise they
needs rescaling. That is the subject of our further study.

Remark that, for any $U_q(\g)$-module $Z$, a singular vector $u=\sum_{i\in I} w_i\tp z_i\in \V\tp Z$,
defines a $U_q(\g_+)$-equivariant map  $V^*\to Z$. Since the $U_q(\g_+)$-module $V^*$ is cyclicly generated by $z_1$, we call it
  generating coefficient of $u$.
\begin{lemma}
\label{dyn_root_vec_zero}
Suppose that $Z$ is generated by highest weight vector $v_\la$.
Suppose that $\check f_{mj}v_\la=0$ for  $m<j $. Then  $\check f_{ij}v_\la=0$ for all $i\leqslant m$.
%ii) If $\phi\sqsubset \check f_{mj}v$ for $\phi\in \C[\Theta]$, then  $\phi\sqsubset \check f_{ij}v$ for all $i\leqslant m$.
\end{lemma}
\begin{proof}
Recall that  for all $k\geqslant m$ and $\mu=\nu_{k}-\nu_{k+1}$,
$
e_\mu \check f_{kj}=\check f_{k+1,j}\bar A^j_k\mod U_q(\g)\g_+ $, cf. \cite{M7}.
Then $\check{f}_{mj}v_\la=0$
implies $\bar A^j_{m,\ldots,k-1}(\la)\check{f}_{kj}v_\la=0$ for all $k\geqslant m+1$.
If follows from (\ref{hat_f}) that
$$
\check{f}_{m-1,j}v_\la=
a_{m}(\la)f_{m-1,m}\check f_{mj}v_\la+\sum_{k=m+1}^ja_{k}(\la)f_{m-1,k}\bigl(\bar{A}^j_{m,\ldots k-1}(\la)\check f_{kj}\bigr)v_\la,
$$
where $a_i(\la)$ are non-vanishing numerical factors.
This implies $\check{f}_{m-1,j}v_\la=0$. Induction on $m$ proves the statement for all $i\leqslant m$.
\end{proof}
\begin{corollary}
The vector $\check u_j\in V\tp Z_\la$ turns zero
i) only if $\bar A_{m}^j(\la) =0$,
ii)  if only if $\check f_{mj}v_\la =0$, for some $m<j$.
\label{sing_vec_zero}
\end{corollary}
\begin{proof}
  "Only if" in both statements follow from the equalities $\check u_{jj}=\bar A_{1,\ldots,j-1}^jv_\la$
  and, respectively, $\check u_{1j}=\check f_{1j}v_\la$.
  "If"  is due to Lemma \ref{dyn_root_vec_zero}.
\end{proof}
\begin{remark}
 Suppose that $Z$ is a family of $U_q(\g)$-modules of highest weight $\la$
ranging in an algebraic set $\Theta\subset \h^*$.
 Then Lemma \ref{dyn_root_vec_zero} admits an obvious modification if one replaces
equality to zero with divisibility by some $\phi\in \C_q[\Theta]$. Assuming it indecomposable,
 Corollary \ref{sing_vec_zero} can be appropriately reworded if $\C_q[\Theta]$ is a unique factorization
 domain and $Z$ has no zero divisors.
 In what follows, we apply this modification to $\Theta=\Theta^\k$ and $Z=M_\la^\k$.
\end{remark}
In the next statement we essentially assume that $\k\not=\h$.
\begin{propn}
\label{zero_M_j^k}
  The submodule $M_j^\k\subset V\tp M^\k_\la$ vanishes for all $j\in \bar I^\k$.
\end{propn}
\begin{proof}
  It is sufficient to consider the case $\Pi_\k=\{\mu\}$.
Choose $\la \in \Theta^\k$ so that  $\check u_j\not=0$. By Proposition \ref{reg_f_1j} below, such weights are dense in $\Theta^\k$.
Let $i\in I$ be such that $\nu_i-\nu_j=\mu$.
As $\bar A^j_i(\la)=0$, for all $k\leqslant i$ one has
$$
\check{f}_{kj}v_\la\simeq
\sum_{k<\vec m< i}f_{k,\vec m,i}\bar{A}^j_{k,\vec m,i}(\la)\check f_{ij}v_\la\in  M_{\la-\mu}\subset  M_{\la} .
$$
Therefore  $\check u_j$ is in the submodule $V\tp M_{\la-\mu}$ and so does $u_j$ for all $\la\in \Theta^\k$.
\end{proof}

\subsection{Projection to $\gr \Vc_\bullet $}
It follows from Proposition \ref{principal_diagram1} that the image of $\hat u_j$ in $\gr \Vc_\bullet $ lies in $\gr \Vc_j$, and thus $\hat u_j=\hat D_j w_\la^j\mod \Vc_{j-1}$ with some $\hat D_j\in \C$. Obviously $\hat D_1= 1$. For higher $j$ calculation gives
$$
\quad \hat D_2\simeq\frac{[\xi_{12}]_q}{[\eta_{12}]_q}, \quad
\hat D_3\simeq\frac{[\xi_{13}]_q}{[\eta_{13}]_q}\frac{[\xi_{23}]_q}{[\eta_{23}]_q},
\quad
\hat D_4\simeq\frac{[\xi_{14}]_q}{[\eta_{14}]_q}\frac{[\xi_{24}]_q}{[\eta_{24}]_q}\frac{[\xi_{34}]_q}{[\eta_{34}]_q},
\quad
\hat D_5\simeq\frac{[\xi_{15}]_q}{[\eta_{15}]_q}\frac{[\xi_{25}]_q}{[\eta_{25}]_q}\frac{[\frac{\xi_{35}}{2}]_q}{[\frac{\eta_{35}}{2}]_q},
$$
$$
\hat D_6\simeq\frac{[\xi_{16}]_q}{[\eta_{16}]_q}\frac{[\xi_{36}]_q}{[\eta_{36}]_q}\frac{[\frac{\xi_{26}}{2}]_q}{[\frac{\eta_{26}}{2}]_q}\frac{[\xi_{56}]_q}{[\eta_{56}]_q} ,\quad
\hat D_7\simeq\frac{[\xi_{27}]_q}{[\eta_{27}]_q}\frac{[\xi_{37}]_q}{[\eta_{37}]_q}\frac{[\frac{\xi_{17}}{2}]_q}{[\frac{\eta_{17}}{2}]_q}\frac{[\xi_{57}]_q}{[\eta_{57}]_q}
\frac{[\xi_{67}]_q}{[\eta_{67}]_q},
$$
where the quantities
$$\xi_{ij}=h_i-h_j+\rho_i-\rho_j+\frac{1}{2}(||\nu_i||^2-||\nu_j||^2) \in \h+\C$$
 are related to eigenvalues of the operator $\Q$ by  $q^{2\xi_{ij}}=x_i\bar x_j$.
For reader's convenience, $\xi_{ij}$ are given in Section \ref{eta_xi}.

Let $\Theta_j\subset \h^*$,  $j>4$,  denote the set of weights such that $q^{2\eta_{4j}(\la)}=-q^{2}$.
\begin{lemma}
 The module $M_j$ is not contained in  $\Vc_{j-1}$ for a dense open subset $\Theta_j^\circ\subset\Theta_j$.
\label{theta_zero}
\end{lemma}
\begin{proof}
Observe that $q^{2\xi_{4j}}=-q^2\not =1$ and $q^{2\xi_{j'j}}=q^8\not =1$. For all other $i<j$, the functions $q^{2\xi_{ij}}$  are not constant on $\Theta_j$.
Therefore all $q^{2\xi_{ij}}$ with $i<j$ are distinct from $1$, off an algebraic subset in $\Theta_j$.
For such weights, $M_j$ cannon be in $\Vc_j$, since the eigenvalue $x_j$ is distinct from
the $\Q$-eigenvalues $\{x_i\}_{i=1}^{j-1}$ on $ \Vc_{j-1}$.
\end{proof}

\begin{corollary}
  For all $j>4$,  $\check f_{1j}$ identically vanishes on  $\Theta_j$.
  \label{regular}
\end{corollary}
\begin{proof}
The product  $\check D_j=\prod_{k=1}^{j-1}\bar A^j_k\hat D_j$ is polynomial in $q^{\pm h_\mu}$ and turns zero on  $\Theta_j$  for $j>4$
thanks to the factor $\bar A^j_{j'}$ and the equality $\frac{\eta_{j'j}}{2}=\eta_{4j}-1$.   The tensor $\check u_j\in V\tp M_\la$ is projected to
$\check C(\la)w^j_\la=0\mod  \Vc_{j-1}$. That is possible only in the following two cases: either $\check u_j$ turns zero or $M_j\subset \Vc_{j-1}$.
By Lemma \ref{theta_zero}, $\check u_j(\la)=0$ on $\Theta_j^\circ$ and hence on $\Theta_j$.
This yields $\check u_{1j}= \check f_{1j}v_\la =0$ on $\Theta_j$ for the generating coefficient.
\end{proof}

\subsection{Regularization of singular vectors in $V\tp M_\la$}
In this section, we evaluate a scalar function $\dt_j\sqsubset\check u_j$ and show that renormalized singular vectors
$\dt_j^{-1}\check u_j$ do not turn zero at all $\la$.

Denote by $J$ the two-sided ideal in $U_q(\g_-)$ generated by the relation $f_\al f_\bt= \bar q^3 f_\bt f_\al$.
The non-zero elements $f_{ij}$ modulo $J$ read
$$
f_{12}=f_\al, \quad f_{23}=[3]_qf_\bt, \quad f_{34}=f_\al, \quad f_{45}=f_\al, \quad f_{56}= [3]_qf_\bt, \quad f_{67}=f_\al,
$$
$$
f_{24}  =(\bar q^{3}-q^3)f_\bt f_\al,
\quad
f_{35}  =\frac{\bar q^{2}-1}{[2]_q}f_\al^2,
\quad
f_{57}  =(\bar q^{3}- q^3)f_\bt f_\al,
$$
$$
f_{25}  =\frac{(\bar q^3- q)(\bar q^3-q^3)}{[2]_q^2}f_\bt f_\al^2.
$$
Introduce $\check g_{ij}$ as  polynomials in $y_1^{\pm 1},\ldots, y_7^{\pm 1}$ with coefficients in $\U_q(\g_-)$
by similar formulas as $\check f_{ij}$ with all $\bar A^j_k$  in (\ref{hat_f}) and (\ref{check_u}) replaced by $\bar A_k=\frac{1-y_k}{q-\bar q}$.
\begin{lemma}
\label{25modJ}
One has $\check g_{35}\simeq f_\al^2\frac{\bar qy_4+q}{q+\bar q}$. Furthermore,  $\check g_{25}\simeq[3]_qf_\bt f_\al^2 \frac{\bar qy_4+ q}{q+\bar q}\mod J$.
\end{lemma}
\begin{proof}
All calculations will be done modulo $J$. First of all,
$$
\check g_{35}\simeq f_{35}\bar A_4+f_{34}f_{45}=f_\al^2(\frac{\bar q^2-1}{[2]_q}\frac{y_4-1}{\bar q-q}+1)=f_\al^2\frac{\bar qy_4+q}{q+\bar q}.
$$
Substitute  $f_{ij}\mod J$ into
$
\check g_{25}\simeq f_{25}\bar A_3\bar A_4+f_{23}f_{35}\bar A_4+f_{24}f_{45}\bar A_3+f_{23}f_{34}f_{45}
$
and get
$$
\check g_{25}\simeq f_\bt f_\al^2\Bigl(\frac{(\bar q^3-q)(\bar q^3-q^3)}{[2]_q^2}\bar A_3\bar A_4+[3]_q\frac{\bar q^{2}-1}{[2]_q}\bar A_4+(\bar q^{3}-q^3)\bar A_3+
[3]_q\Bigr).
$$
Computation of the coefficient in the brackets completes the proof.
\end{proof}
Set $\dt_j=\frac{q^{\eta_{4j}-1}+q^{-\eta_{4j}+1}}{q+\bar q}=\frac{q^{\frac{\eta_{j'j}}{2}}+q^{-\frac{\eta_{j'j}}{2}}}{q+\bar q}$ for $5\leqslant j$ and $\dt_j=1$ for $j=1,2,3,4$. Corollary \ref{regular} assures
that  $\check f_{1j}/\dt_j$ is a polynomial in $q^{\pm h_\mu}$ via identification $y_i=q^{2\eta_{ij}}$, $i<j$.
\begin{propn}
\label{reg_f_1j}
For all $j>1$, the vectors $\check f_{1j}/\dt_j$  do not turn zero on $\h^*$.
\end{propn}
\begin{proof}
Again, all calculations are done modulo $J$.
The statement is trivial for $j=2$. For $j=3$, it follows from the factorization
$\check f_{13}=f_{12}f_{23}$. We also have
$$
\check f_{14}\simeq f_{12}\check f_{24}\simeq f_{\al}([3]_qf_{\bt}f_\al+(\bar q^3-q^3)f_{\bt}f_\al\bar A_3^4)
=[3]_qf_{\al}f_{\bt}f_\al q^{\eta_{35}} \not =0,
$$
which proves the case $j=4$.

For each $j\geqslant 5$  and all $i<j$ we assign $y_i=q^{2\eta_{ij}}$ and use Lemma \ref{25modJ}:
the key point is that $\dt_j$ cancels the factor $\bar q y_4+q$ in all cases.
Modulo $J$, we have
$
\check f_{15}(\la)\simeq f_{12}\check g_{25}
$
and
$
\check f_{16}(\la)\simeq f_{12}\check g_{25}f_{56}
$.
This implies the statement for $j=5,6$. Finally,
$$
\check f_{17}(\la)\simeq f_{12} \check g_{27}\simeq f_{12}(\check g_{26}f_{67}+\check g_{25}f_{57}\bar A_6)
\simeq f_{12} (\check g_{25}f_{56}f_{67}+\check g_{25}f_{57}\bar A_6)
=f_{\al} \check g_{25} f_\bt f_\al[3]_q y_6,
$$
which proves it for $j=7$.
\end{proof}
\subsection{Decomposition of $V\tp M_\la$}
\label{decomposition_Verma}
Denote by $u_j$ the singular vectors $\check u_j/\delta_j(\la)$ for all $j\in I$. Then $u_j=D_jw^\la_j \mod \Vc_{j-1}$
with $D_j\simeq \prod_{i=1}^{j-1}\phi_{ij}$, where $\phi_{ij}=[\xi_{ij}]_q$ if $i\not= j'$ and $\phi_{j'j}=[\frac{\xi_{j'j}}{2}]_q$.

\begin{lemma}
\label{M_j_in_M_i}
 The submodule $M_j$ is contained in $M_i$ with  $i<j$ if and only if $\phi_{ij}(\la) =0$.
 %In other words,  if and only if  $q^{2\xi_{ij}(\la)}=1$ for $i\not =j'$ or   $q^{\xi_{ij}(\la)}=1$ for $i= j'$.
\end{lemma}
 \begin{proof}
The eigenvalues $q^{2(\rho,\nu_1)}x_j$ of the operator $q^{2(\rho,\nu_1)}\Q$ on the submodules $M_j$ are
\be
q^{2(\la,2\al+\bt)+10}, \> q^{2(\la,\al+\bt)+8}, \> q^{2(\la,\bt)+2}, \> q^{-2},
\> q^{-2(\la,\bt)-2},  \> q^{-2(\la,\al+\bt)-8}, \> q^{-(\la,2\al+\bt)-10},
\label{eigenvalues}
\ee
counting from the left.
If $\phi_{ij}(\la)$ and hence $D_j(\la)$ turns zero, then $x_j=x_i$ and $M_j\subset \Wc_{j-1}$, by Proposition \ref{W=V}.
Suppose that $x_j$ is distinct from $x_k$ if $k<j$ and $k\not =i$. As follows from (\ref{eigenvalues})
such weights are dense in the set of solution to $\phi_{ij}(\la)=0$. Then $M_j$ can lie only in $M_i$ and hence it does
for all such weights.

Conversely, let $\la$ be so that $M_{j}\subset M_i$. Then $[\xi_{ij}]_q=0$ and  hence  $\phi_{ij}(\la)=0$ if $i\not =j'$.
If $i=j'$, we can assume that $M_j\not \subset M_i$ for $i\not =j'$ since $\la$ is in the closure of such weights,
cf. (\ref{eigenvalues}).   Proposition \ref{W=V} then suggests that $D_j(\la)=0$, and the only vanishing factor can  be $\phi_{j'j}$.
Then it is true for all $\la$.
 \end{proof}
As an application of the obtained results, we describe direct sum decomposition of the module $V\tp M_\la$.
This corresponds to maximal conjugacy classes, with $\k=\h$.
Fix  $t=\diag(t_i)\in T$ and choose $\la$ to fulfill the condition $tq^{-2h_\rho}=q^{2h_\la}$.
This fixes the relation between the entries of $t$ and the $\Q$-eigenvalues as $t_i=q^{2(\la+\rho,\nu_i)}=x_iq^{2\dt_{i4}-2\rho_1}$.
%$x_j=q^{2(\la+\rho,\nu_j)+(\nu_j,\nu_j)-(\nu_1,\nu_1)+2(\rho,\nu_1)}$
\begin{propn}
  Suppose that $t\in T^\h$ and parameterize it as in (\ref{t-h-reg}) and (\ref{t-h-bord}). Then $V\tp M_\la=\op_{i=1}^7M_i$ if and only if
$q^{-2}\not =x,y,xy$ for regular $t$ and $q^{-4}\not=x^2$ for borderline $t$.
\end{propn}
\begin{proof}
The sum $\sum_{i=1}^7M_i$ exhausts all of $V\tp M_\la$ if and only if it is direct, by Proposition \ref{W=V}
or, equivalently, if $D^\h=\prod_{i=1}^7D_i$ is not zero. Explicitly,
$$
D^\h\simeq \prod_{i<j\atop i\not =j',4}(x_i-x_j)\prod_{i=1}^3(x_i-q^{2\rho_1})\simeq
\prod_{i<j\atop {i\not =j' \atop i,j\not =4}}(t_i-t_j)\prod_{i=1}^3(t_i-q^{-2})\prod_{i=1}^3(t_i-1)
\simeq \prod_{i=1}^3(t_i-q^{-2})
$$
for $t\in T^\h$.
 This implies  the stated conditions on $q$ guaranteeing $D^\h\not=0$
(mind that $q$ is not a root of unity).
\end{proof}
Next we essentially assume that $\k\not =\h$ and  put
$\phi^\k_j=\prod_{i\in \bar I^\k_j}\phi_{ij}$,  $j\in I^\k$.
\begin{lemma}
For every $j\in I^\k$, the vector $\pi^\k_V(u_j)$ vanishes once $\phi^\k_j(\la)=0$. If $\k$ is of type $\k_s$ or $\k_l$, then
$\pi^\k_V(u_j)$ is  divisible by $\phi_j^\k$.
\end{lemma}
\begin{proof}
 By Lemma \ref{M_j_in_M_i}, $u_j\in M_i$ once $\phi^\k_j(\la)=0$. On the other hand, $\pi^\k_V(u_j)=0$ if $\la\in \Theta^\k$,
  by Proposition \ref{zero_M_j^k}.
Therefore $\pi^\k_V(u_j)$ vanishes in $V\tp M_\la^\k$. If the semisimple part of $\k$ has rank $1$, trigonometric polynomials
on $\Theta^\k$ form a principal ideal domain as $\dim \Theta^\k=1$. Therefore $\pi^\k_V(u_j)$ is divisible by $\phi^\k_j(\la)$.

\end{proof}

\section{Module structure of  $V\tp M^{\k_s}_\la$}
\label{secModStrS}
Throughout this section $\nu\in \Rm^+$ is a short root and $\k=\k_s$
is the reductive subalgebra
of maximal rank with the root system $\{\pm  \nu\}$.
We aim to prove that the vectors $u_j^\k=\frac{1}{\phi^\k_j(\la)}\pi^\k_V(u_j)\in V\tp M^{\k}_\la$
are regular functions of $\la$ and do not vanish at all weights. Their projections to $\gr \Vc_\bullet^\k$ are equal to $D^\k_jw^\la_j$ with $D^\k_j\simeq\prod_{i\in I^\k_j}\phi_{ij}$.
\subsection{Regularization of singular vectors in $V\tp M^{\k}_\la$}
Assuming $j\in I^\k$, denote  by $c^\k_j$   the coefficient  in the expansion
$u^\k_j \simeq w_j\tp c^\k_jv_\la + \ldots $, where the suppressed terms belong to $\sum_{i<j}w_i\tp M^{\k}_\la$.
It is equal to $\frac{\prod_{i<j}[\eta_{ij}]}{d_j d_j^\k}$ up to a non-vanishing factor.
\begin{table}[h]
  \centering
\begin{tabular}{|c||c|c|c|}
  \hline
 \multicolumn{4}{|c|}{ $\nu=\al$, $(\la,\al)=0$,  $\theta=(\la,\bt)$}
 \\\hline
  $j\in I_\k$&$\phi^\k_j$&$D^\k_j$&$c^\k_j$\\\hline\hline
   $1$&$1$&$1$&$1$\\\hline
   $3$&$\scriptstyle[\theta+3]_q$&$\scriptstyle[\theta+3]_q$&$\scriptstyle[\theta]_q$\\\hline
   $6$&$\scriptstyle[\theta+4]_q[\theta+3]_q[\theta+3]_q$&$\scriptstyle[2\theta+9]_q[\theta+5]_q$&$\scriptstyle[\theta+2]_q[2\theta+6]_q$\\
  \hline
   \multicolumn{4}{|c|}{$\nu=\al+\bt$, $q^{2(\la,\nu)+6}=1$, $\theta=(\la,\al)$}  \\
\hline
  $j\in I_\k$&$\phi^\k_j$&$D^\k_j$&$c^\k_j$\\\hline\hline
   $1$&$\scriptstyle 1$&$\scriptstyle 1$&$\scriptstyle 1$\\\hline
   $2$&$\scriptstyle 1$&$\scriptstyle[\theta+1]_q$&$\scriptstyle[\theta]_q$\\\hline
   $5$&$\scriptstyle [\theta+1]_q[\theta]_q$&$\scriptstyle[2\theta+3]_q[\theta+2]_q$&$\scriptstyle[\theta-1]_q[2\theta]_q$\\
  \hline
   \multicolumn{4}{|c|}{$\nu=2\al+\bt$, $q^{2(\la,\nu)+8}=1$, $\theta=(\la,\al)$}  \\
\hline
   $j\in I_\k$&$\phi^\k_j$&$D^\k_j$&$c^\k_j$\\\hline\hline
   $1$&$\scriptstyle 1$&$\scriptstyle 1$&$\scriptstyle 1$\\\hline
   $2$&$\scriptstyle 1$&$\scriptstyle[\theta+1]_q$&$\scriptstyle[\theta]_q$\\\hline
   $3$&$\scriptstyle 1$&$\scriptstyle[\theta]_q[2\theta+1]_q$&$\scriptstyle[\theta+1]_q[2\theta+4]_q$\\
  \hline
 \end{tabular}
    \caption{Type $\k\simeq \k_s$}
    \label{short_table}
\end{table}
\begin{propn}
  The vectors $u_j^\k$ do not vanish at all weights.
\end{propn}
\begin{proof}
We  parameterize  $\Theta^\k\subset \h^*$  with the complex variable $\theta$ as in Table \ref{short_table}.

The case $\nu=\al$ follows from the fact that the null sets of $D^\k_j$ and $c^\k_j$ do not intersect, cf.  Table \ref{short_table}.

In the case of $\nu=\al+\bt$, the statement is trivial for $j=1$ and immediate for
    $u_2^\k\simeq -q^{-1}w_1\tp f_\al v_\la+w_2\tp [\theta]_qv_\la $.
  Let us prove it for $j=5$.
 Choose a basis $f_\bt f_\al^3v_\la$,  $f_\al f_\bt f_\al^2 v_\la$, $f_\al^2 f_\bt f_\al v_\la$
 in the weight subspace $M_\la^\k[\la-\bt-3\al]$.
The coefficient in $\check f_{15}/(\dt_5\phi_5^\k)$ corresponding to $f_\bt f_\al^3$
  is equal to $[\eta_{25}]_q=[\theta+1]_q$, up to an invertible multiplier. Hence $u_5^\k$ does not turn zero unless $[\theta+1]_q=0$. However, $D_5^\k\not =0$
at such $\la$. Therefore $u_5^\k\not =0$ at all weights from $\Theta^\k$.

Finally, consider the case $\nu=2\al+\bt$.
   The projection $M_\la\to M_\la^\k$ is an isomorphism on subspaces of  weights $\la-\mu$ with $\mu<2\al+\bt$.
   So are the weights of the generating coefficients in $u^\k_j=\pi^\k_V(u_j)$. They  do not turn zero as $u_j\not=0$ at all weights. This completes the proof.
\end{proof}

\subsection{Decomposition of $\V\tp M_\la^{\k_s}$}
Fix $t\in T^{\nu}$ and let $\k$ be the stabilizer of $t$. Choose $\la \in \Theta^\k$ from the equality $q^{2\la+2\rho}=tq^{\nu}$.
We use $x\in \C^\star$,  $x^2\not=1$,  to parameterise the spectrum $\{x^{\pm 1},1\}$  of  $t$ as  (\ref{t-s}).
\begin{propn}
  The module $V\tp M_\la^\k$ splits into the  direct sum $\op_{j\in I^\k}M_j^\k$ if and only if $q\not =x^{\pm 1}$.
 \end{propn}
\begin{proof}
  Table \ref{short_table} gives $D^\k=\prod_{j\in I^\k}D_j^\k\simeq (q-x)(q-x^{-1})$.
  The sum $\sum_{j\in I^\k}M_j^\k$ exhausts all of $V\tp M_\la^\k$ if and only if $D^\k$ does not vanish.
  It also implies that the eigenvalues $qx^{\pm 1}, q^2$ of $q^{10}\Q$ are pairwise distinct, hence the sum is direct.
\end{proof}

\section{Module structure of   $V\tp M^{\k_l}_\la$}
\label{secModStrL}
In this section $\nu\in \Rm^+$ is one of the three long roots, and $\k\simeq \k_l\subset \g$ is the reductive subalgebra
of maximal rank with the root system $\{\pm  \nu\}$.
Now $\Theta^\k$ the set of weights $\la$ that satisfy the condition $q^{2(\la+\rho,\nu)-6}=1$. We parameterize it
with the complex variable $\theta =(\la,\al)$.
There are two pairs $(l,k)\in P(\nu)$,  so $\#\bar I_{\k}=2$ and $\#I_{\k}=5$.
\subsection{Regularization of singular vectors for quasi-Levi $\k$}
The case of  quasi-Levi $\k_l$  turns out to be simpler, so we consider it first.
For $\nu=3\al+2\bt$, we have $I^\k=\{1,2,3,4,5\}$, $q^{2(\la, 3\al+2\bt)+12}=1$, and
$\phi_j^\k=1$ for all $j\in I^\k$. For
$\nu=3\al+\bt$, we have $I^\k=\{1,2,3,4,6\}$, $q^{2(\la, 3\al+\bt)+6}=1$,  and
$\phi_j^\k=1$ for $j\in I^\k$ apart from
$\phi_6^\k=[3\theta]_q$, where $\theta =(\la,\al)$.

 For each  $\mu <\nu$, projection $M_\la[\la-\mu]\to M_\la^\k[\la-\mu]$ is an isomorphism. Therefore
  $\pi^\k_V(u_j)=u_j^\k$ does not turn zero unless maybe $\nu=3\al+\bt$, $j=6$.  In the
 latter case $\C^\k_6=[3\theta-3]_q[2\theta-1]_q[\theta-2]_q[2\theta]_q$ and $c^\k_6=[2\theta+1]_q[\theta-1]_q[2\theta]_q[3\theta+3]_q$,
 so $u_6^\k$ may vanish only at $q^{2\theta}=\pm 1, q^2$.
  However,  $u_{36}^\k\simeq [2\theta+1]_q\check f_{36}v_\la$, and  one can easily check that
$$
\check f_{36}=\frac{[[f_\bt,f_\al]_{q},f_\al]_{q^3}}{[2]_q^2}\bar A_4\bar A_5+f_{\al}[f_\al,f_\bt]_{q^3}\bar A_5
+f_{\al}^2f_{\bt}(\frac{q+q^{-4\theta-1}}{[2]_q})[3]_q\not =0
$$
at all weights. Hence
  $u_6^\k\not= 0$ at all weights.
\subsection{Regularization of singular vectors for Levi $\k$}
\begin{table}[h]
  \centering
\begin{tabular}{|c||c|c|c|c|c|}
  \hline
  $j\in I_\k$&$\phi^\k_j$&$\check D_j/(\phi^\k_j\dt_j)$&$\check c_{j}/(\phi^\k_j\dt_j)$&$D^\k_j$\\\hline\hline
   $1$&$\scriptstyle 1$&$\scriptstyle 1$&$\scriptstyle 1$&$\scriptstyle 1$\\\hline
   $2$&$\scriptstyle 1$&$\scriptstyle [\theta+1]_q$&$\scriptstyle [\theta]_q$&$\scriptstyle [\theta+1]_q$\\\hline
   $4$&$\scriptstyle[\theta+2]_q$&$\scriptstyle [2\theta+6]_q[\theta+5]_q$&$\scriptstyle \frac{[2\theta+4]_q}{[\theta+2]_q}[\theta+3]_q[\theta]_q$&$\scriptstyle \frac{[2\theta+6]_q}{[\theta+3]_q}[\theta+5]_q$\\\hline
   $5$&$\scriptstyle[\theta+1]_q$&$\scriptstyle [3\theta+6]_q[2\theta+5]_q[\theta]_q$&$\scriptstyle \frac{[3\theta+3]_q}{[\theta+1]_q}[\theta-1]_q[2\theta+4]_q[\theta]_q$&$\scriptstyle \frac{[3\theta+6]_q}{[\theta+2]_q}[2\theta+5]_q[\theta]_q$\\\hline
     $7$&$\scriptstyle [3\theta+6]_q[\theta+1]_q$&$\scriptstyle [2\theta+5]_q[3\theta+9]_q[2\theta+4]_q[\theta+4]_q$&$\scriptstyle [2\theta +3]_q\frac{[3\theta+3]_q}{[\theta+1]_q}[2\theta+4]_q[\theta+3]_q[\theta]_q$
     &$\scriptstyle [2\theta+5]_q\frac{[3\theta+9]_q}{[\theta+3]_q}\frac{[2\theta+4]_q}{[\theta+2]_q}[\theta+4]_q$\\
  \hline
\end{tabular}
    \caption{$\nu=\bt$, $(\la,\bt)=0$,  $(\la,\al)=\theta$}
\end{table}
From this table we conclude that $u_j/\phi^\k_j$ may be divisible by following factors:
\begin{enumerate}
  \item $[\theta+3]_q\sqsubset \bar A^4_2$, for $j=4$,
  \item  $[\theta+2]_q\sqsubset\bar A^5_2$, $ [\theta]_q\sqsubset\bar A^5_4$, for $j=5$,
  \item $[2\theta+4]_q\sqsubset\bar A^7_4$, $[\theta+3]_q\sqsubset\bar A^7_5$, for $j=7$.
\end{enumerate}
Introduce $\psi_j^\k$ for $j\in I^\k$ as
$$
\psi_1^\k=1,\quad \psi_2^\k=1,\quad \psi_4^\k=[\theta+3]_q,\quad \psi_5^\k=[\theta+2]_q,\quad \psi_7^\k=[\theta+2]_q[\theta+3]_q.
$$
\begin{propn}
For all $j\in I^\k$, the singular vector $\pi^\k_V(u_j)/\phi^\k_j$ is divisible by $\psi^\k_j$, and $u^\k_j=\pi^\k_V(u_4)/(\phi^\k_j\psi^\k_j)\not =0$  at all $\la$.
\end{propn}
\begin{proof}
There is nothing to prove for $j=1,2$,  so we assume $j=4,5,7$.
We should only check divisibility, because non-zero property then follows from factorization of $\check D_j$.
One can check that
\be
\check f_{14}v_\la&\simeq & (f_{13}f_{\al}-f_\al f_\bt f_\al\bar q^3)[\theta+3]_q[\theta+2]_q,
\label{f14}\\
\check f_{25}v_\la
&\simeq &\Bigl(f_{24}f_{\al}\bar q^2\frac{1-q^{2\theta-2}}{q-q^{-1}}+f_{\bt}f_\al^2\frac{1}{[2]}(\bar q[3]_q+\frac{1-q^{2\theta}}{q-q^{-1}})\Bigr)\dt_5[\theta+2]_qv_\la,
\label{f25}
\\
\check f_{47}v_\la&\simeq &(f_{46}f_{\al} q^3-f_\al f_\bt f_\al)[\theta+3]_q[\theta+2]_q,
\label{f47}
\ee
in the module $M_\la^\k$.
Now the proof for $j=4$ readily follows from (\ref{f14}).

Furthermore,  (\ref{f25}) implies that $\check f_{15}$ is not divisible by $[\theta]_q$ and divisible by $[\theta+2]_q$, by Lemma \ref{dyn_root_vec_zero}.
Therefore, $\check f_{15}/(\dt_5[\theta+1]_q[\theta+2]_q)$ is regular and never turns zero. This proves the case $j=5$.

Equality (\ref{f47}) implies that $\check f_{i7}$ are divisible by  $[\theta+2]_q[\theta+3]_q$ for all $i\leqslant  4$.
Since $[\theta+2]_q \sqsubset \bar A^7_2 \sqsubset \phi^\k_7$, we have $[\theta+2]_q^2 \sqsubset \bar A^7_2 \check f_{37}v_\la \simeq e_\bt \check f_{27}v_\la$ and need to prove
$[\theta+2]_q^2 \sqsubset  \check f_{27}$.
It is easy to check that the operator
$$M_\la^\k[\la-3\al-2\bt]=\Span \{f_\bt f_\al^i f_\bt f_\al^{3-i} v_\la\}_{i=0}^2
\stackrel{e_\bt}{\longrightarrow} \Span \{f_\al^i f_\bt f_\al^{3-i} v_\la\}_{i=0}^2=M_\la^\k[\la-3\al-\bt]$$
is  independent of $\la$ and has zero kernel. Therefore  $\check f_{27}$ is divisible by $[\theta+2]_q^2$ as required.
\end{proof}
Define $D^\k_j$ from the equality $\wp^\k_j(u^\k_j)=D^\k_jw^\la_j$ and put $D^\k=\prod_{j\in I^\k}D^\k_j$.
 It follows that $D^\k_j\simeq D_j/(\phi^\k_{j}\psi^\k_j)$.
%\begin{propn}
%  The module $V\tp M_\la^\k$ splits into direct sum $\op_{i\in I^\k_j}M^\k_j$ if and only if $D^\k(\la)\not =0$.
%\end{propn}
\subsection{Decomposition of $\V\tp M_\la^{\k_l}$}
Fix $t\in T^{\nu}$ with the stabilizer $\k$ and determine  $\la\in \Theta^\k$
from the equality $q^{2h_\la+2h_\rho}=tq^{h_\nu}$.
We use $x\in \C^\star$,  $x^2,x^3\not=1$,  to parameterise the spectrum $\{x^{\pm2},x^{\pm 1}, 1\}$ of  $t$ as in (\ref{t-l-reg}) and (\ref{t-l-bord}),
for the regular and borderline cases. Then the spectrum of the operator $q^{2\rho_1}\Q$ on the module $\V\tp M_\la^{\k_l}$ is $\{x^{\pm 2},q^{3}x^{\pm 1},q^{-2}\}$.
\begin{propn}
  The module $V\tp M_\la^\k$ splits into  direct sum $\op_{j\in I^\k}M_j^\k$ if and only if
 a) $q^3\not=x^{\pm 1}$, $\frac{q^3-x^{\pm 3}}{q-x^{\pm 1}}\not =0$,
 for $\nu=\bt$, b) $q^3\not=x^{\pm 1}, x^{\pm 3}$ for  $\nu\not =\bt$.
 \end{propn}
\begin{proof}
Assuming $t\in T^\nu$,  we get
  $D^\k\simeq
(x-q^3)^2(x-\bar q^3)^2
\frac{x^3-q^3}{x-q}
\frac{\bar x^3-q^3}{\bar x-q}
$, for $\nu=\bt$
and $D^\k\simeq
(x-q^3)^2(x-\bar q^3)^2
(x^3-q^3)(\bar x^3-q^3)
$, for  $\nu\not=\bt$.
  The sum $\sum_{j\in I^\k}M_j^\k$ exhausts all of $V\tp M_\la^\k$ if and only if $D^\k\not=0$.

Now suppose that $D^\k\not =0$. Then  the sum $\sum_{j\in I^\k}M_j^\k$ is direct for $\nu=\bt$ by Proposition \ref{W=V}.
For $\nu\not =\bt$, that is obvious if the $\Q$-eigenvalues are pairwise distinct (which is violated
for a finite number of $q$). Still it is true in all cases.
We give a sketch of the proof based on character analysis.
One can check that $M_j^\k$ are quotients of $\tilde M_j^\k$, where $\tilde M_j^\k=M_{\la+\nu_j}/M_{\la+\nu_j-\ell_j\nu}$ with $\ell_j=\frac{2(\nu_j,\nu)}{(\nu,\nu)}+1\in \{1,2\}$.
It is easy to see that $\sum_{j\in I^\k}^{}\Char \tilde M_j^\k=\Char V\times \frac{e^{\la}(1-e^{-\nu})}{\prod_{\al\in \Rm^+}(1-e^{-\al})}=\Char (V\tp M^\k_\la)$.
Now the map $\op_{j\in I^\k}\tilde M_j^\k\to \V\tp M_\la^{\k_l}$ is injective because it is surjective.
This also implies $\tilde M_j^\k\simeq M_j^\k$ for all $j\in I^\k$.
\end{proof}

\section{Decomposition of $\V\tp M_\la^{\k}$ for pseudo-Levi $\k$}
\label{secDecPL}
Let $\mu$ and $\nu$ denote, respectively, the minimal and maximal weight in $\Pi_\k$ and put $\m\subset \k$ to be the reductive subalgebra
of maximal rank such that  $\Pi_\m=\{\mu\}$. For  $\la\in \Theta_\m\subset \Theta_\k$, the
 homomorphism  $M_\la\to M_\la^\k$ factors through the projection  $M_\la^\m \to M_\la^\k$.
Its restriction  $M_\la^\m[\la-\xi]\to M_\la^\k[\la-\xi]$
is an isomorphism for  $\xi <\nu$. It follows from here that the  map $\sum_{\nu_i<\nu}w_i\tp M_\la^\m[\la-\nu_i]\to \sum_{\nu_i<\nu}w_i\tp M_\la^\k[\la-\nu_i]$
 sends non-vanishing singular vectors $u_j^\m$ with $j\in I^\k\subset I^\m$ over to non-zero singular vectors, $u^\k_j
\in V\tp M_\la^\k$.  One can check that $D_j^\k(\la)=D_j^\m(\la)\not=0$, for all $j\in I^\k$.

\begin{table}[h]
  \centering
\begin{tabular}{|c||c|r|r|c|c|}
  \hline
$\Pi_\k$&$t$&$q^{2(\la,\al)}$& $q^{2(\la,\bt)}$& $I^\k$\\\hline\hline
$\{\bt,3\al+\bt\}$&$(e^{\frac{2\pi i}{3}},e^{- \frac{2\pi i}{3}},e^{- \frac{2\pi i}{3}})$&$e^{-\frac{2\pi i}{3}}q^{-2}$&$1$&$\{1,2,4\}$\\
  \hline
$\{\bt,3\al+\bt\}$&$(e^{-\frac{2\pi i}{3}},e^{\frac{2\pi i}{3}},e^{\frac{2\pi i}{3}})$& $e^{\frac{2\pi i}{3}}q^{-2}$ & $1$&$\{1,2,4\}$\\
  \hline
$\{\al,3\al+2\bt\}$&$(-1,-1,1)$&$1$ & $-q^{-6}$&$\{1,3\}$\\
  \hline
$\{\al+\bt,3\al+\bt\}$&$(-1,1,-1)$&$-1$ & $-q^{-6}$&$\{1,2\}$\\
  \hline
$\{\bt,2\al+\bt\}$&$(1,-1,-1)$&$-q^{-4}$ & $1$&$\{1,2\}$\\
  \hline

\end{tabular}
    \caption{Pseudo-parabolic type}
\label{non-Levi}
\end{table}
\begin{propn}
  For all pseudo-Levi $\k\in \g$, $V\tp M_\la^\k=\op_{j\in I^\k}M_j^\k$.
\end{propn}
\begin{proof}
Since $D^\k\not =0$, the  sum $\sum_{j\in I^\k}M_j^\k$ gives all $V\tp M_\la^\k$. It is direct as the $\Q$-eigenvalues are distinct.
Indeed, for $\k=\k_{l,l}$ we have
$$
q^{2\xi_{12}}=e^{\mp \frac{2\pi i}{3}}, \quad q^{2\xi_{14}}=e^{\pm \frac{2\pi i}{3}}q^8, \quad q^{2\xi_{24}}=e^{\mp \frac{2\pi i}{3}}q^{6},
$$
where the upper sign corresponds to the first row in Table \ref{non-Levi}.
For the three $\k_{s,l}$-points we have
$$
q^{2\xi_{12}}=-q^{2}, \quad q^{2\xi_{12}}=-q^{2}, \quad q^{2\xi_{12}}=-q^{-2},
$$
respectively, from the top downward.
\end{proof}
\appendix
\section{Appendix}
\subsection{Formulas for $\eta_{ij}$ and $\xi_{ij}$}
\label{eta_xi}
Below we present the explicit expressions for $\eta_{ij}$ and $\xi_{ij}$,  $i\leqslant j$, arranging them into matrices.
$$
(\eta_{ij})=
\left(
\begin{array}{lllllllll}
  0 & h_{\al} & h_{\al+\bt}+3 & h_{2\al+\bt}+4 & h_{3\al+\bt}+3 & h_{3\al+2\bt}+6 & h_{4\al+2\bt}+6 \\
   & 0 & h_{\bt} & h_{\al+\bt}+3 & h_{2\al+\bt}+4 & h_{2\al+2\bt}+4 & h_{3\al+2\bt}+6, \\
   &  & 0 & h_{\al} & h_{2\al}-2 & h_{2\al+\bt}+4 & h_{3\al+\bt}+3 \\
   &  &  & 0 & h_{\al} & h_{\al+\bt}+3 & h_{2\al+\bt}+4 \\
   &  &  &  & 0 & h_{\bt} & h_{\al+\bt}+3 \\
   &  &  &  &  & 0 & h_{\al} \\
   &  &  &  &  &  &0
\end{array}
\right)
$$
$$
(\xi_{ij})=
\left(
\begin{array}{lllllllll}
   0& h_{\al}+1 & h_{\al+\bt}+4 & h_{2\al+\bt}+6 & h_{3\al+\bt}+6 & h_{3\al+2\bt}+9 & h_{4\al+2\bt}+10 \\
   & 0 & h_{\bt}+3 & h_{\al+\bt}+5 & h_{2\al+\bt}+5 & h_{2\al+2\bt}+8 & h_{3\al+2\bt}+9, \\
   &  & 0 & h_{\al}+2 & h_{2\al}+2 & h_{2\al+\bt}+5 & h_{3\al+\bt}+6 \\
   &  &  & 0 & h_{\al} & h_{\al+\bt}+3 & h_{2\al+\bt}+4 \\
   &  &  &  & 0 & h_{\bt}+3 & h_{\al+\bt}+4 \\
   &  &  &  &  & 0 & h_{\al}+1 \\
   &  &  &  &  &  & 0
\end{array}
\right)
$$
\subsection{Entries of the matrix $F$}
\label{ShapInv}
Here we present explicit expressions of the entries $f_{ij}$, $i<j$, participating in the reduced Shapovalov inverse form.
$$
\scriptstyle
f_{12}=f_\al, \quad f_{23}=[3]_qf_\bt, \quad f_{34}=f_\al, \quad f_{45}=f_\al, \quad f_{56}= [3]_qf_\bt, \quad f_{67}=f_\al,
$$
$$
\scriptstyle
f_{13}  =[f_\bt,f_\al]_{q^3},
\quad
f_{24}  =[f_\al,f_\bt]_{q^3},
\quad
f_{35}  =\frac{\bar q^{2}}{[2]_q}[f_\al,f_\al]_{q^2},
\quad
f_{46}  =[f_\bt,f_\al]_{q^3},
\quad
f_{57}  =[f_\al,f_\bt]_{q^3},
$$
$$
\scriptstyle
f_{14}=\frac{\bar q[f_\al,[f_\bt,f_\al]_{q^3}]_{q^3}}{[2]_q},
\quad
f_{25}  =\frac{[f_\al,[f_\al,f_\bt]_{q}]_{q^3}}{[2]_q^2},
\quad
f_{36}  =\frac{[[f_\bt,f_\al]_{q},f_\al]_{q^3}}{[2]_q^2},
\quad
f_{47} =\frac{\bar q[[f_\al,f_\bt]_{q^3},f_\al]_{q^3}}{[2]_q}
$$
$$
\scriptstyle
f_{15}  =\frac{\bar q^2[f_\al,[f_\al,[f_\bt,f_\al]_{q^3}]_{q^3}]_{ q}}{[2]_q^2}
\quad
f_{26} =\frac{\bar q^{3}[f_\bt,[f_\al,[f_\al,f_\bt]_{q}]_{q^3}]_{q^6}}{[2]_q^2},
\quad
f_{37}  =\frac{\bar q^2[[[f_\al,f_\bt]_{q^3},f_\al]_{q^3},f_\al]_{q}}{[2]_q^2},
$$
$$
\scriptstyle
f_{16}  =\frac{\bar q^2}{[2]_q^2}[f_\bt,[f_\al,[f_\al,[f_\bt,f_\al]_{q^3}]_{q^3}]_{q}]_{q^3},
\quad
f_{27}  =\frac{\bar q^2}{[2]_q^2}[[[[f_\al,f_\bt]_{q^3},f_\al]_{q^3},f_\al]_{q},f_\bt]_{q^3},
$$
$$
\scriptstyle
 f_{17} =\frac{\bar q^2}{[2]_q^2}[f_\al,[f_\bt,[f_\al,[f_\al,[f_\bt,f_\al]_{q^3}]_{q^3}]_{q}]_{q^3}]_{q}.
$$

\end{document}